\documentclass[a4paper, 11pt, reqno]{amsart}

\usepackage{amssymb}
\usepackage[margin=2.5cm]{geometry}

\newtheorem{theorem}{Theorem}[section]
\newtheorem{lemma}[theorem]{Lemma}
\newtheorem{lem}[theorem]{Lemma}
\newtheorem{conj}[theorem]{Conjecture}

\newtheorem{claim}[theorem]{Claim}
\newtheorem{obs}[theorem]{Observation}

\theoremstyle{definition}

\theoremstyle{remark}

\newcommand{\ind}{ind}

\newcommand{\Pp}{\mathcal{P}}

\numberwithin{equation}{section}

\begin{document}

\title{On the inducibility of cycles}


\author{Dan Hefetz}
\address{Department of Computer Science, Ariel University, Ariel 40700,
Israel}
\curraddr{}
\email{danhe@ariel.ac.il}
\thanks{}

\author{Mykhaylo Tyomkyn}
\address{School of Mathematics, Tel Aviv University, Tel Aviv 69978, Israel}
\curraddr{}
\email{tyomkynm@post.tau.ac.il}
\thanks{The second author is supported in part by ERC Starting Grant 633509}


\date{\today}

\begin{abstract}
In 1975 Pippenger and Golumbic proved that any graph on $n$ vertices admits at most $2e(n/k)^k$ induced $k$-cycles. This bound is larger by a multiplicative factor of $2e$ than the simple lower bound obtained by a blow-up construction. Pippenger and Golumbic conjectured that the latter lower bound is essentially tight. In the present paper we establish a better upper bound of $(128e/81) \cdot (n/k)^k$. This constitutes the first progress towards proving the aforementioned conjecture since it was posed.
\end{abstract}

\maketitle

\section{Introduction}

A common theme in modern extremal combinatorics is the study of densities or induced densities of fixed objects (such as graphs, digraphs, hypergraphs, etc.) in large objects of the same type, possibly under certain restrictions. This general framework includes Tur\'an densities of graphs and hypergraphs, local profiles of graphs and their relation to quasi-randomness, and more. One such line of research was initiated by Pippenger and Golumbic~\cite{PG}. Given graphs $G$ and $H$, let $D_H(G)$ denote the number of induced subgraphs of $G$ that are isomorphic to $H$ and let $I_H(n) = \max\{D_H(G) : |G| = n\}$. A standard averaging argument was used in~\cite{PG} to show that the sequence $\{I_H(n)/\binom{n}{|H|}\}_{n=|H|}^{\infty}$ is monotone decreasing, and thus converges to a limit $\ind(H)$, the so-called \emph{inducibility} of $H$. 

Since it was first introduced in 1975, inducibility has been studied in many subsequent papers. Determining this invariant seems to be a very hard problem. To illustrate the current state of knowledge (or lack thereof), it is worthwhile to note that even the inducibility of paths of length at least $3$ and cycles of length at least $6$ are not known. Still, the inducibility of a handful of graphs and graph classes is known. These include various very small graphs (see, e.g.,~\cite{Baletal, EL, Hst}) and complete multipartite graphs (see, e.g.,~\cite{BEH, BNT, BS}). Additional recent results on inducibility can be found, e.g., in~\cite{HHN, Hu, MS}. Some of the recent progress in this area is due to Razborov's theory of flag algebras~\cite{Raz}, which provides a framework for systematic computer-aided study of questions of this type. 

While, trivially, the complete graph $H = K_k$ and its complement achieve the maximal possible inducibility of $1$, the natural analogous  question, which graphs on $k$ vertices \emph{minimise} the quantity $\ind(H)$, which has been asked in~\cite{PG}, is still open.

Let $H$ be an arbitrary graph on $k$ vertices, where $k$ is viewed as large but fixed. By considering a balanced blow-up of $H$ (and ignoring divisibility issues), it is easy to see that $\ind(H) \geq k!/k^k$. An iterated blow-up construction provides only a marginally better lower bound of $k!/(k^k-k)$. Pippenger and Golumbic~\cite{PG} conjectured that the latter is tight for cycles. 

\begin{conj} [\cite{PG}] \label{conj:exact}
$ind(C_k) = k!/(k^k - k)$ for every $k \geq 5$.
\end{conj}

\noindent
Note that the requirement $k \geq 5$ appearing in Conjecture~\ref{conj:exact} is necessary. Indeed, $ind(C_3) = 1$ since $C_3 = K_3$ is a complete graph and, as shown in~\cite{PG}, $ind(C_4) = 3/8$ since $C_4 = K_{2,2}$ is a balanced complete bipartite graph. The authors of~\cite{PG} also posed the following asymptotic version of the above conjecture.  
\begin{conj} [\cite{PG}] \label{conj:asympt}
$ind(C_k) = (1+o(1)) k!/k^k$.
\end{conj}

In support of Conjecture~\ref{conj:asympt}, it was shown in~\cite{PG} that $I_{C_k}(n) \leq \frac{2n}{k} \left(\frac{n-1}{k-1} \right)^{k-1}$ holds for every $k \geq 4$. This implies that $ind(C_k) \leq 2e \cdot k!/k^k$, leaving a multiplicative gap of $2e$ (which is approximately $5.4366$) between the known upper and lower bounds. In this paper we partially bridge the above gap by proving a better upper bound on the inducibility of $C_k$, namely $ind(C_k) \leq (128/81)e\cdot k!/k^k$ (note that $(128/81)e$ is approximately $4.2955$).

\begin{theorem} \label{thm:main}
For every $k\geq 6$ we have
$$
\ind(C_k) \leq \frac{128e}{81} \cdot \frac{k!}{k^k}.
$$ 
\end{theorem}

We note that the case $k=5$ of Conjecture~\ref{conj:exact} was settled by Balogh, Hu, Lidick\'{y} and Pfender~\cite{Baletal}, who showed, in particular, that if $n$ is a power of $5$, then $I_{C_5}(n)$ is uniquely attained by the iterated blow-up of $C_5$. The proof which was given in~\cite{Baletal} combines flag algebras~\cite{Raz} and stability methods. It is also worth noting that, in \emph{triangle-free} graphs, all pentagons are induced. Maximising the number of pentagons in triangle-free graphs is an old problem of Erd\H{o}s~\cite{ErdC5}, which was solved recently, using flag algebras, by Grzesik~\cite{Gr} and independently by Hatami, Hladk\'{y}, Kr\'{a}l', Norine and Razborov~\cite{Many} (prior to the use of flag algebras, the best result was due to Gy\H{o}ri~\cite{Gy} who gave an elegant elementary proof of a slightly weaker bound). 

\medskip

The rest of this paper is organized as follows. In the next section we collect some basic properties of graphs which maximise the number of induced $k$-cycles, and recall the proof of the bound in~\cite{PG}. In Section~\ref{sec::largeMinDegree} we prove Theorem~\ref{thm:main} for graphs with large minimum degree. Section~\ref{sec::main} constitutes the main part of our proof of Theorem~\ref{thm:main}. In order to improve the presentation of the paper, the proof of Claim~\ref{cl:programA} is postponed to the appendix.
 
\section{Preliminaries} \label{sec:prelim}

In this section we establish a number of lemmas which will pave the way to the proof of our main result later on. In particular, we will present a slightly modified version of the proof of the upper bound on $I_{C_k}(n)$ from~\cite{PG}. We begin by introducing some notation and terminology which will be used throughout the paper. Some of this notation is standard and can be found, e.g., in~\cite{Byellow}. For a positive integer $n$, we denote by $[n]$ the set $\{1, \ldots, n\}$. 
For a graph $G = (V,E)$, let $\overline{G} = (V, \overline{E})$, where $\overline{E} = \{xy : x, y \in V, x \neq y, xy \notin E\}$, denote the complement of $G$. For a set $S \subseteq V$, let $G[S]$ denote the subgraph of $G$ induced by $S$. For a vertex $v \in V$, let $N_G(v) = \{w \in V : vw \in E\}$ denote the neighbourhood of $v$ and let $d_v = |N_G(v)|$ denote the degree of $v$. The minimum degree of $G$, denoted by $\delta(G)$, is $\min \{d_v : v \in V\}$. As a less standard piece of notation, let $x_{uw} = |N_G(u) \cap N_G(w)|$ denote the co-degree of two vertices $u, w \in V$ and let $z_{uvw} = |N_G(u) \cap N_G(v) \cap N_G(w)|$ denote the co-degree of three vertices $u, v, w \in V$. For graphs $H$ and $G$, and vertices $v_1, \ldots, v_{\ell} \in V(G)$, let $D_H(G, v_1, \ldots, v_{\ell})$ denote the number of (unlabeled) induced copies of $H$ in $G$ containing $v_1, \ldots, v_{\ell}$. To simplify notation, we abbreviate $D_{C_k}$ to $D_k$.

Throughout this paper we reserve the letter $f$ to denote the function $f(x) = x e^{-x}$ on the domain $[0,\infty)$. The following basic analytic properties of $f$ will be used repeatedly in our proofs. 
\begin{obs} \label{obs:fx}
$f(x)$ is monotone increasing on $[0,1]$ and monotone decreasing on $[1,\infty)$. Consequently, $f(x)$ attains its global maximum at $x=1$. Moreover, $f(x)$ is concave on $[1,2]$. 
\end{obs}

Next, we state a lemma which is implicit in~\cite{PG} and provides a counting principle which will play a crucial role in the proof of Theorem~\ref{thm:main}. Since the statement of the lemma is somewhat technical, we first explain informally what it will be used for. Suppose that we wish to count induced $k$-cycles in some graph $G$. The way we do this (following~\cite{PG}) is by building a copy of $C_k$ vertex by vertex. Suppose we have already built a path $u_1, u_2, \ldots, u_r$ and now wish to extend it by adding another vertex $u_{r+1}$ (assume $r+1 < k$). This vertex should be a neighbour of $u_r$ but should not be adjacent to $u_i$ for any $1 \leq i \leq r-1$. That is, when choosing $u_{r+1}$, we consider the neighbours of $u_r$ but exclude from this set all the vertices that were ``considered'' before (neighbours of $u_1$, neighbours of $u_2$, etc.). We aim to establish an upper bound on the number of induced $C_k$'s we will be able to build via this process.  

In order to state the lemma formally, we will need some additional notation and terminology. Suppose that, for some positive integers $n$ and $k$, we have a family of injective functions $\Pp \subseteq \{p : [k] \rightarrow [n]\}$. Let $A_0 = \emptyset$ and, for every $1 \leq i < k$, let $A_i = \{(a_1, \ldots, a_i) \in [n]^i : a_r \neq a_s \textrm{ for every } 1 \leq r < s \leq i\}$. For every $1 \leq i < k$ and every $\bar{a} = (a_1, \ldots, a_i) \in A_i$, let $\Pp_{\bar{a}} = \{p \in \Pp : p(1) = a_1, \ldots, p(i) = a_i\}$, let $X_{\bar{a}}=X_{\bar{a}}(\Pp) = \{m \in [n]: m \neq p(j+1) \textrm{ for every } 0 \leq j \leq i-1 \textrm{ and every } p \in \Pp_{(a_1, \ldots, a_j)}\}$, and let $\mathcal{Q}_{\bar{a}} =\mathcal{Q}_{\bar{a}}(\Pp) = \{p \in [n]^k : p(j) \in X_{\bar{a}} \textrm{ for every } i < j \leq k\}$.

\begin{lem} [\cite{PG}] \label{lem:excl}
For positive integers $k$ and $n$, let $\Pp \subseteq \{p : [k] \rightarrow [n]\}$ be a family of injective functions. Suppose that $\Pp$ satisfies the following `exclusion property': for every $1 \leq i < k$ and every $\bar{a} = (a_1, \ldots, a_i) \in A_i$ we have $\Pp_{\bar{a}} \subseteq \mathcal{Q}_{\bar{a}}$. Then $|\Pp| \leq (n/k)^k$.
\end{lem}
%
\noindent
Let us now give a slightly modified proof of the upper bound on $I_{C_k}(n)$ from~\cite{PG}.
\begin{lem} [\cite{PG}] \label{lem:ub2e}
Let $n \geq k \geq 4$ be integers, then 
$$
I_{C_k}(n) \leq 2 e \frac{n^k}{k^k}.
$$
\end{lem}

\begin{proof}
Let $G$ be an arbitrary graph on $n$ vertices and let $v \in V(G)$ be an arbitrary vertex. We will use Lemma~\ref{lem:excl} in order to bound $D_k(G,v)$ from above as follows. Label the vertices of $C_k$ by $1, 2, 4, \ldots, k, 3$ along the cycle (note the unusual order). By counting all labeled induced embeddings $\phi$ of $C_k$ (labeled as above) into $G$, subject to $\phi(1) = v$, we obtain twice the number of induced $k$-cycles containing $v$, as each of them will be counted once for each `direction'. We have at most $d_v^2$ choices for the images of $2$ and $3$. Since $\phi$ maps $C_k$ to an induced $k$-cycle of $G$, for any choice of $\phi(2)$ and $\phi(3)$ the choices for $\phi(i)$ where $4 \leq i \leq k$ form a family which satisfies the exclusion property on the ground set $V(G) \setminus (\{v\} \cup N_G(v))$ of size $n - d_v - 1$. Applying Lemma~\ref{lem:excl} then yields 
\begin{align} \label{eq:PGvertex}
D_k(G,v) \leq \frac{1}{2} d_v^2 \left(\frac{n - d_v - 1}{k-3}\right)^{k-3} \leq \frac{1}{2} d_v^2 \left(\frac{n - d_v}{k-3}\right)^{k-3}.
\end{align}

\noindent
Standard calculations show that the last expression is maximised at $d_v = 2n/(k-1)$. Hence
\begin{align} \label{eq:PGcalculus}
D_k(G,v) \leq 2 \left(\frac{n}{k-1}\right)^{k-1} \leq 2 e  \frac{n^{k-1}}{k^{k-1}}.
\end{align}
Since $v$ was arbitrary, it follows that 
$$
D_k(G) = \frac{1}{k} \sum_{v \in V(G)} D_k(G,v) \leq 2 e \frac{n^k}{k^k}.
$$  
Finally, since $G$ was arbitrary, we conclude that 
$$
I_{C_k}(n) \leq 2 e \frac{n^k}{k^k}.
$$
\end{proof}

\medskip

Our next lemma asserts that in a graph $G$ which maximises $D_k$ over all $n$-vertex graphs every vertex is contained in approximately the same number of induced $k$-cycles.

\begin{lem} \label{lem:homogen}
Let $k \geq 5$ be an integer, let $G$ be a graph on $n$ vertices which maximises $D_k$ and suppose that $D_{k}(G) = \alpha n^k/k^k$. Then, for every $v \in V(G)$, we have 
$$
D_k(G,v) = \alpha \frac{n^{k-1}}{k^{k-1}} + O(n^{k-2}).
$$
\end{lem}

\begin{proof}
Since $\sum_{v \in V(G)} D_k(G,v)$ counts each induced $k$-cycle in $G$ precisely $k$ times, it follows that 
$$
\frac{1}{n} \sum_{v \in V(G)} D_k(G,v) = \frac{k}{n} D_k(G) = \alpha \frac{n^{k-1}}{k^{k-1}}.
$$   
Let $v^+$ be a vertex of $G$ which is contained in the largest number of induced $k$-cycles, and let $v^-$ be a vertex of $G$ which is contained in the smallest number of induced $k$-cycles. Let $G'$ be obtained from $G$ by Zykov's symmetrisation~\cite{Zy}, i.e., remove $v^-$ and add a twin of $v^+$ instead (the two copies of $v^+$ in $G'$ are not connected by an edge). Then 
$$
D_k(G) \geq D_k(G') = D_k(G) - D_k(G,v^-) + D_k(G,v^+) - D_k(G,v^-,v^+),
$$ 
where the inequality above follows from our assumption that $G$ maximises $D_k$ and the equality holds since, for $k \geq 5$, no induced $k$-cycle in $G$ can contain both copies of $v^+$. Therefore 
$$
D_k(G,v^+) - D_k(G,v^-) \leq D_k(G,v^-,v^+) \leq \binom{n}{k-2}.
$$
Hence
$$
D_k(G,v^+) \leq \alpha \frac{n^{k-1}}{k^{k-1}} + O(n^{k-2}) \;\;\; \textrm{ and } \;\;\; D_k(G,v^-) \geq \alpha \frac{n^{k-1}}{k^{k-1}} - O(n^{k-2}).
$$
\end{proof}

By Lemma~\ref{lem:homogen}, in order to prove Theorem~\ref{thm:main} it suffices to show that every $n$-vertex graph $G$ has some vertex $v$ such that $D_k(G,v) \leq (128/81) e \cdot (n/k)^{k-1}$. As noted above, the maximum of~\eqref{eq:PGvertex} is attained when $d_v = 2n/(k-1)$. We will need the following lemma, which states that $D_k(G,v) \leq (128/81) e \cdot (n/k)^{k-1}$ holds if $d_v$ is sufficiently far from $2n/(k-1)$. For every vertex $u \in V(G)$ let $c_u = k d_u/n$.

\begin{lemma} \label{lem:rangec}
Let $G$ be a graph on $n$ vertices and let $u \in V(G)$. If $c_u \leq 1$ or $c_u \geq 4$, then 
$$
D_k(G,u) \leq \frac{128 e}{81} \cdot \left(\frac{n}{k}\right)^{k-1}.
$$ 
\end{lemma}  

\begin{proof}
By~\eqref{eq:PGvertex} we have 
\begin{align*}
D_k(G,u) &\leq \frac{1}{2} d_u^2 \left(\frac{n - d_u}{k-3}\right)^{k-3} = \frac{1}{2} \frac{n^2}{k^2} c_u^2 \left(\frac{k n - c_u n}{k(k-3)}\right)^{k-3}
= \frac{1}{2} \frac{n^{k-1}}{k^{k-1}} c_u^2 \left(\frac{k-c_u}{k-3}\right)^{k-3} \\
&= \frac{1}{2} \frac{n^{k-1}}{k^{k-1}} c_u^2 \left(1 + \frac{3-c_u}{k-3}\right)^{k-3}
\leq \frac{1}{2} \frac{n^{k-1}}{k^{k-1}} c_u^2 e^{3 - c_u}, 
\end{align*}
where the last inequality holds since $1 + x \leq e^x$ for every $x \in \mathbb{R}$.

Now, if $c_u \leq 1$, then 
\begin{align*}
\frac{1}{2} \frac{n^{k-1}}{k^{k-1}} c_u^2 e^{3 - c_u} \leq \frac{1}{2} \frac{n^{k-1}}{k^{k-1}} e^2 < \frac{128 e}{81} \cdot \frac{n^{k-1}}{k^{k-1}}.
\end{align*}
where the first inequality holds since $x^2 e^{3-x}$ is increasing in the interval $[0, 1]$.

\noindent Similarly, if $c_u \geq 4$, then 
$$
\frac{1}{2} \frac{n^{k-1}}{k^{k-1}} c_u^2 e^{3 - c_u} \leq \frac{1}{2} \frac{n^{k-1}}{k^{k-1}} \cdot \frac{4^2}{e} < \frac{128 e}{81} \cdot \frac{n^{k-1}}{k^{k-1}},
$$ 
where the first inequality holds since $x^2 e^{3-x}$ is decreasing for $x \geq 4$.
\end{proof}

\section{Induced $k$-cycles in graphs with a large minimum degree} \label{sec::largeMinDegree}
In this section we prove that graphs with a large minimum degree cannot contain too many induced $k$-cycles. 

\begin{lemma} \label{lem:mindeg2}
Let $G$ be a graph on $n$ vertices and let $v \in V(G)$ be a vertex of minimum degree. If $c_v \geq 2$, then 
$$
D_k(G,v) \leq \frac{128 e}{81} \cdot \left(\frac{n}{k}\right)^{k-1}.
$$ 
\end{lemma}

\begin{proof} 
For every vertex $w \in N_G(v)$, we can bound $D_k(G,v,w)$ from above as follows. As before, label the vertices of $C_k$ by $1, 2, 4, \ldots, k, 3$ along the cycle. We will upper bound the number of labeled induced embeddings $\phi$ of $C_k$ (labeled as above) into $G$, subject to $\phi(1) = v$ and $\phi(2) = w$. We have at most $d_v - x_{vw}$ choices for $\phi(3)$ and at most $d_w - x_{vw}$ choices for $\phi(4)$. For each such choice of $\phi(3)$ and $\phi(4)$ the choices of $\phi(i)$ for every $5 \leq i \leq k$ form a family which satisfies the exclusion property on the ground set $V(G) \setminus (N_G(v) \cup N_G(w))$, which has size $n - d_v - d_w + x_{vw}$. Hence, by Lemma~\ref{lem:excl}, we have
\begin{align} \label{eq:PGedgeGen}
D_k(G,v,w) \leq (d_v - x_{vw}) (d_w - x_{vw}) \left(\frac{n - d_v - d_w + x_{vw}}{k-4}\right)^{k-4}.
\end{align}
Let $c = c_v$ and, for every $w \in N_G(v)$, let $x_w = k x_{vw}/n$. Then
\begin{eqnarray}\label{eq:normalise}
D_k(G,v,w) &\leq & (d_v - x_{vw}) (d_w - x_{vw}) \left(\frac{n - d_v - d_w + x_{vw}}{k-4}\right)^{k-4}\nonumber \\ 
&=& \frac{n^2}{k^2} (c - x_w) (c_w - x_w) \left(\frac{k n - c n - c_w n + x_w n}{k(k-4)}\right)^{k-4}\nonumber \\ 
&=& \frac{n^{k-2}}{k^{k-2}} (c - x_w) (c_w - x_w) \left(1 + \frac{4 - c - c_w + x_w}{k-4}\right)^{k-4}\nonumber \\
&\leq & \frac{n^{k-2}}{k^{k-2}} (c - x_w) (c_w - x_w) e^{4 - c} e^{-(c_w - x_w)}.
\end{eqnarray}
Note that
$$
D_k(G,v) = \frac{1}{2} \sum_{w \in N_G(v)} D_k(G,v,w),
$$
where the factor $1/2$ is due to the fact that $v$ has precisely two neighbours in every $k$-cycle which contains it. Therefore 
\begin{align} \label{eq:vertexsum}
D_k(G,v) \leq \frac{n^{k-2}}{k^{k-2}} \cdot \frac{1}{2} e^{4-c} \sum_{w \in N_G(v)} (c - x_w) (c_w - x_w) e^{-(c_w - x_w)}.
\end{align}
Viewing $c$ as a parameter, we define a two variable real function 
$$
g_c(x_w, c_w) = (c - x_w) (c_w - x_w) e^{-(c_w - x_w)}.
$$ 
By Lemma~\ref{lem:rangec} we can assume that $c_w \leq 4$ holds for every $w \in N_G(v)$. Since, moreover, $d_v = \delta(G)$ by assumption, it follows that $0 \leq x_w \leq c \leq c_w \leq 4$ for every $w \in N_G(v)$. Thus, abandoning the graph structure, we can upper bound the right hand side of~\eqref{eq:vertexsum} as follows. 
\begin{equation} \label{eq:maximumXmCm}
D_k(G,v) \leq \frac{n^{k-1}}{k^{k-1}} \cdot \frac{1}{2} e^{4-c} c (c - x_m) (c_m - x_m) e^{-(c_m - x_m)},
\end{equation}
where $(x_m, c_m)$ is a point at which $g_c(x_w, c_w)$ attains its global maximum on $I_c = [0,c] \times [c,4]$ (note that such a point exists as $I_c$ is compact and $g_c$ is continuous; note also that it need not be unique).

Suppose first that $(x_m, c_m)$ lies in the interior of $I_c$. Differentiating with respect to $x_w$ yields 
\begin{equation} \label{eq::partialDerivativeX}
\frac{\partial g_c}{\partial x_w}(x_m,c_m) = g_c(x_m, c_m) \left(1 - \frac{1}{c - x_m} - \frac{1}{c_m - x_m}\right).
\end{equation}
Comparing~\eqref{eq::partialDerivativeX} to zero we obtain $1/(c - x_m) + 1/(c_m - x_m) = 1$ which, in particular, implies that $c_m - x_m > 1$. On the other hand, differentiating with respect to $c_w$ yields
\begin{equation} \label{eq::partialDerivativeC}
\frac{\partial g_c}{\partial c_w}(x_m, c_m) = g_c(x_m, c_m) \left(\frac{1}{c_m - x_m} - 1\right).
\end{equation}
Comparing~\eqref{eq::partialDerivativeC} to zero results in the contradiction $c_m - x_m = 1$. It follows that the point $(x_m, c_m)$ lies on the boundary of $I_c$. 

If $x_m = c$, then clearly $g_c(x_m, c_m) = 0$. Suppose then that $x_m = 0$. Then, by~\eqref{eq:maximumXmCm}, we have
\begin{align*}
\frac{k^{k-1}}{n^{k-1}}
D_k(G,v) \leq \frac{1}{2} c^2 e^{4-c} c_m e^{-c_m} \leq \frac{1}{2} c^3 e^{4-2c} \leq \frac{1}{2} 2^3 e^{4-2\cdot 2} =4 < \frac{128}{81} e,
\end{align*}
where the second inequality holds by Observation~\ref{obs:fx} since $c_m \geq c \geq 2$ by assumption, 
and the third inequality holds since $h(x) = x^3 e^{4-2x}/2$ is decreasing for $x \geq 2$.   
Finally, consider the case $x_m \in (0,c)$. In this case
$$
\frac{\partial g_c}{\partial x_w}(x_m, c_m) = 0
$$
and thus 
\begin{equation} \label{eq:partialx0}
1/(c - x_m) + 1/(c_m - x_m) = 1.
\end{equation}
This, in turn, implies that the (one-sided) partial derivative $\frac{\partial g_c}{\partial c_w}(x_m, c_m)$ is negative. It thus follows that $c_m = c$ as otherwise there would exist some $c < c' < c_m$ such that $g_c(x_m, c') > g_c(x_m, c_m)$ contrary to the maximality of $g_c(x_m, c_m)$. Therefore, $c_m - x_m = c - x_m = 2$ holds by~\eqref{eq:partialx0}. Hence, by~\eqref{eq:maximumXmCm}, we have
$$
\frac{k^{k-1}}{n^{k-1}} D_k(G,v) \leq \frac{1}{2} e^{4-c} c \cdot 4 e^{-2} = 2 c e^{2-c} \leq 4 < \frac{128e}{81}, 
$$
where the second inequality holds by Observation~\ref{obs:fx} since $c \geq 2$ by assumption. 
\end{proof}

\section{Proof of the main result} \label{sec::main}

In this section we will prove Theorem~\ref{thm:main}. Let $G$ be an arbitrary graph on $n$ vertices, let $v \in V(G)$ be an arbitrary vertex and let $u,w \in N_G(v)$ be two non-adjacent vertices in the neighbourhood of $v$. We will use Lemma~\ref{lem:excl} in order to bound $D_k(G,u,v,w)$ from above as follows. Label the vertices of $C_k$ by $1, 2, 4, 6, 7, \ldots, k-1, k, 5, 3$ along the cycle. By counting all labeled induced embeddings $\phi$ of $C_k$ (labeled as above) into $G$, subject to $\phi(1) = v, \phi(2) = u$ and $\phi(3) = w$ we obtain precisely the number of induced $k$-cycles containing $v,u$ and $w$. 

We have at most $d_u - x_{uv} - x_{uw} + z_{uvw}$ and $d_w - x_{vw} - x_{uw} + z_{uvw}$ choices for the images of $4$ and $5$, respectively. Since $\phi$ maps $C_k$ to an induced $k$-cycle of $G$, the choices of $\phi(i)$ for every $6 \leq i \leq k$ form a family which satisfies the exclusion property on the ground set $V(G) \setminus (N_G(v)\cup N_G(u)\cup N_G(w))$ of size at most $n - d_u - d_v - d_w + x_{uv} + x_{vw} + x_{uw} - z_{uvw}$. Applying Lemma~\ref{lem:excl} then yields

\begin{align*} 
D_k(G,u,v,w) &\leq (d_u - x_{uv} - x_{uw} + z_{uvw}) (d_w - x_{vw} - x_{uw} + z_{uvw}) \nonumber\\ 
&\cdot\left(\frac{n - d_u - d_v - d_w + x_{uv} + x_{vw} + x_{uw} - z_{uvw}}{k-5}\right)^{k-5}.
\end{align*}

Let $A_v = \{(u,w) \in N_G(v) \times N_G(v) : u \neq w \textrm{ and } uw \notin E(G)\}$ denote the set of ordered pairs of non-adjacent neighbours of $v$. Then
$$
D_k(G,v) = \frac{1}{2} \sum_{(u,w) \in A_v} D_k(G,u,v,w).
$$ 

Recall that $c_u = k d_u/n$, $c_w = k d_w/n$ and $c = k d_v/n$. Moreover, let $\bar{x}_{uv} = k x_{uv}/n$, $\bar{x}_{vw} = k x_{vw}/n$, $\bar{x}_{uw} = k x_{uw}/n$ and $\bar{z}_{uvw} = k z_{uvw}/n$. Then, similarly to~\eqref{eq:normalise} and~\eqref{eq:vertexsum}, we obtain 
\begin{eqnarray} \label{eq:Cherry} 
D_k(G,v) \leq \frac{n^{k-3}}{k^{k-3}} \cdot \frac{1}{2} e^{5-c} \sum_{(u,w) \in A_v} && \big[ (c_u - \bar{x}_{uv}-\bar{x}_{uw}+\bar{z}_{uvw}) (c_w - \bar{x}_{vw}- \bar{x}_{uw}+\bar{z}_{uvw}) \nonumber \\ 
&\cdot& e^{-(c_u + c_w - \bar{x}_{uv}- \bar{x}_{vw}- \bar{x}_{uw}+\bar{z}_{uvw})} \big].
\end{eqnarray}
In what follows we will no longer work with $D_k(G)$ itself, but will instead prove upper bounds on the above expression under some fairly general conditions. This is formally stated in the following lemma.

\begin{lemma} \label{lem:program1}
Let $G = (V,E)$ be a graph on $n$ vertices, satisfying $n/k \leq \delta(G) < 2n/k$. Let $v \in V$ be a vertex of minimum degree. Suppose that for every ordered pair $(u,w) \in A_v$ we are given a triple $(c_u^{uw}, c_w^{uw},x^{uw})$ of real numbers such that $c \leq c_u^{uw}, c_w^{uw} \leq 4, x^{uw} \geq \bar{z}_{uvw}, c_u^{uw} - \bar{x}_{uv} - x^{uw} + \bar{z}_{uvw} \geq 0$ and  $c_w^{uw} - \bar{x}_{vw} - x^{uw} + \bar{z}_{uvw} \geq 0$.
Then we have
\begin{eqnarray} \label{eq:program1}
\frac{1}{2} e^{5-c} \sum_{(u,w) \in A_v} && \big[(c_u^{uw} - \bar{x}_{uv} - x^{uw} + \bar{z}_{uvw}) (c_w^{uw} - \bar{x}_{vw}- x^{uw} + \bar{z}_{uvw}) \nonumber \\
&\cdot& e^{-(c_u^{uw} + c_w^{uw} - \bar{x}_{uv}- \bar{x}_{vw}- {x}^{uw} + \bar{z}_{uvw})} \big] \leq \frac{n^2}{k^2} \cdot \frac{128e}{81}.
\end{eqnarray}
\end{lemma}
\noindent
Before proving Lemma~\ref{lem:program1}, we will quickly show how it implies Theorem~\ref{thm:main}. 

\begin{proof} [Proof of Theorem~\ref{thm:main}]
Suppose $G$ maximises $D_k$ over all graphs on $n$ vertices, and let $v \in V(G)$ be such that $d_v = \delta(G)$. If $d_v < n/k$, then $D_k(G,v) \leq (128/81) e \cdot (n/k)^{k-1}$ holds by Lemma~\ref{lem:rangec}. Similarly, if $d_v \geq 2n/k$, then $D_k(G,v) \leq (128/81) e \cdot (n/k)^{k-1}$ holds by Lemma~\ref{lem:mindeg2}. On the other hand, if $n/k\leq d_v< 2n/k$, then, by Lemma~\ref{lem:rangec}, the assumptions of Lemma~\ref{lem:program1} hold with $(c_u^{uw}, c_w^{uw},x^{uw})=(c_u,c_w,\bar{x}_{uw})$. Applying \eqref{eq:Cherry} and Lemma~\ref{lem:program1} we obtain $D_k(G,v) \leq (128/81) e \cdot (n/k)^{k-1}$. In either case, by Lemma~\ref{lem:homogen}, we conclude that 
$$
D_k(G) \leq (1+o(1))(128/81)e \cdot (n/k)^{k}.
$$ 
Hence, by definition of $G$, we have 
$$
I_{C_k}(n) \leq (1+o(1))(128/81)e \cdot (n/k)^{k}.
$$
Normalising and passing to the limit then yields $\ind(C_k)\leq (128/81)e \cdot (k!/k^k)$, as claimed. 
\end{proof}

\begin{proof}[Proof of Lemma~\ref{lem:program1}] 
To simplify notation, we will abbreviate $\bar{x}_{uv}, \bar{x}_{vw}$ and $\bar{z}_{uvw}$ to $x_u, x_w$ and $z_{uw}$, respectively. Fix an arbitrary pair of vertices $(u,w)\in A_v$ and, viewing $c, x_u, x_w$ and $z_{uw}$ as parameters (satisfying $c \geq x_u, x_w \geq z_{uw}$), define a three variable real function 
$$
g_{uw}(x, c_u, c_w) = (c_u - x_{u} - x + z_{uw}) (c_w - x_{w} - x + z_{uw}) e^{-(c_u + c_w - x_{u} - x_w - x + z_{uw})}.
$$ 
Without loss of generality we may assume that $(x^{uw}, c_u^{uw}, c_w^{uw})$ is a point at which $g_{uw}(x, c_u, c_w)$ attains its global maximum on the compact domain
$$
I_{uw} = \{(x, c_u, c_w) : c \leq c_u, c_w \leq 4, x \geq z_{uw}, c_u - x_u - x + z_{uw} \geq 0, c_w - x_w - x + z_{uw} \geq 0\}.
$$ 

Suppose first that $(x^{uw}, c_u^{uw}, c_w^{uw})$ lies in the interior of $I_{uw}$. Differentiating $g_{uw}$ with respect to $x$ yields
$$
\frac{\partial g_{uw}}{\partial x}(x^{uw}, c_u^{uw}, c_w^{uw}) = g_{uw}(x^{uw}, c_u^{uw}, c_w^{uw}) \left(1 - \frac{1}{c_u^{uw} - x_u - x^{uw} + z_{uw}} - \frac{1}{c_w^{uw} - x_w - x^{uw} + z_{uw}}\right).
$$
On the other hand, differentiating $g_{uw}$ with respect to $c_u$ and $c_w$ yields
\begin{equation*} 
\frac{\partial g_{uw}}{\partial c_u}(x^{uw}, c_u^{uw}, c_w^{uw}) = g_{uw}(x^{uw}, c_u^{uw}, c_w^{uw}) \left(\frac{1}{c_u^{uw} - x_u - x^{uw} + z_{uw}} - 1\right)
\end{equation*}
and
\begin{equation*} 
\frac{\partial g_{uw}}{\partial c_w}(x^{uw}, c_u^{uw}, c_w^{uw}) = g_{uw}(x^{uw}, c_u^{uw}, c_w^{uw}) \left( \frac{1}{c_w^{uw} - x_w - x^{uw} + z_{uw}} - 1\right).
\end{equation*}
Comparing all three partial derivatives to zero results in a contradiction. It follows that the point $(x^{uw}, c_u^{uw}, c_w^{uw})$ lies on the boundary of $I_{uw}$. 

Clearly, we cannot have $c_u^{uw} - x_u - x^{uw} + z_{uw} = 0$ or $c_w^{uw} - x_w - x^{uw} + z_{uw} = 0$, as then $g_{uw}(x^{uw}, c_u^{uw}, c_w^{uw}) = 0$. Suppose that $x^{uw} \in (z_{uw}, c_u^{uw} - x_u + z_{uw}) \cap (z_{uw}, c_w^{uw} - x_w + z_{uw})$. In this case, the partial derivative $\frac{\partial g_{uw}}{\partial x}$ exists and equals zero. Hence,
\begin{equation} \label{eq:thecrunch}
1/(c_u^{uw} - x_u - x^{uw} + z_{uw}) + 1/(c_w^{uw} - x_w - x^{uw} + z_{uw}) = 1.
\end{equation}
This, in turn, implies that the partial derivatives $\frac{\partial g_{uw}}{\partial c_u}(x^{uw}, c_u^{uw}, c_w^{uw})$ and $\frac{\partial g_{uw}}{\partial c_w}(x^{uw}, c_u^{uw}, c_w^{uw})$ are negative. It thus follows that $c_u^{uw} = c$ as otherwise there exists some $c < c'_u < c_u^{uw}$ such that $g_{uw}(x^{uw}, c'_u, c_w^{uw}) > g_{uw}(x^{uw}, c_u^{uw}, c_w^{uw})$ contrary to the maximality of $g_{uw}(x^{uw}, c_u^{uw}, c_w^{uw})$. Similarly, $c_w^{uw} = c$. Combined with~\eqref{eq:thecrunch}, this implies that $c - x_u - x^{uw} + z_{uw} \geq 2$ or $c - x_w - x^{uw} + z_{uw} \geq 2$. In either case we have $c \geq 2$, contrary to our assumption that $d_v < 2n/k$. 

The only remaining case is when $x^{uw} = z_{uw}$ for every $(u,w) \in A_v$. In this case we have $g_{uw}(x^{uw}, c_u^{uw}, c_w^{uw}) = f(c_{u}^{uw} - x_u) f(c_{w}^{uw} - x_w)$ (recall that $f(x) = x e^{-x}$). We will treat this case in the following lemma.  
\end{proof}

\begin{lemma} \label{lem:program2}
Let $G = (V,E)$ be a graph on $n$ vertices, satisfying $n/k \leq \delta(G) < 2n/k$. Let $v \in V$ be a vertex of minimum degree. Suppose that for every ordered pair $(u,w) \in A_v$ we are given a pair $(c_u^{uw}, c_w^{uw})$ of real numbers such that $\max \{x_u, c\} \leq c_u^{uw} \leq 4$ and $\max \{x_w, c\} \leq c_w^{uw} \leq 4$. Then  
\begin{equation} \label{eq:program2}
\frac{1}{2} e^{5-c} \sum_{(u,w) \in A_v} f(c_u^{uw} - x_u) f(c_w^{uw} - x_w) \leq \frac{n^2}{k^2} \cdot \frac{128e}{81}.
\end{equation}
\end{lemma}

\begin{proof}
Fix an arbitrary vertex $u \in N_G(v)$. Without loss of generality we may assume that, for every $w$ with $(u,w) \in A_v$, $c_u^{uw}$ is a point at which $f(c_u - x_u)$ attains its global maximum on the closed interval $[\max \{x_u, c\}, 4]$. It readily follows from Observation~\ref{obs:fx} that $c_u^{uw} = \max \{x_u + 1, c\}$. Similarly $c_w^{uw} = \max \{x_w + 1, c\}$ holds for every $w \in N_G(v)$.

Let $L_G = \{u \in N_G(v) : x_u > c-1\}$ and let $S_G = N_G(v) \setminus L_G$. In light of the previous paragraph, the left hand side of~\eqref{eq:program2} equals
\begin{align} \label{eq:splitLS}
P_1(G,v) &:= e^{5-c} \left(X_G/2 + Y_G + Z_G/2 \right),
\end{align}
where 
$$
X_G = \sum_{\stackrel{(u,w) \in A_v}{u, w \in L_G}} e^{-2},
$$
$$
Y_G = \sum_{\stackrel{(u,w) \in A_v}{u \in L_G, w \in S_G}} e^{-1} f(c-x_w),
$$
and
$$
Z_G = \sum_{\stackrel{(u,w) \in A_v}{u, w \in S_G}} f(c - x_u) f(c - x_w).
$$
We may also assume that $(G,v)$ is such that $P_1(G, v)$ is maximal among all pairs $(G, v)$ which satisfy all the assumptions of Lemma~\ref{lem:program2}. 

A vertex $z \in L_G$ will be called \emph{borderline} if decreasing its co-degree $x_{zv}$ by $1$ would result in it being moved to $S_G$, otherwise $z$ will be called \emph{internal}. 
For an edge $uw \in E(G[L_G])$ (if such an edge exists), let $G^+_{uw}$ be any graph which is obtained from $G$ by deleting the edge $uw$ and adding two new edges $u y_1$, $w y_2$, where $y_1, y_2 \in V(G) \setminus (\{v\} \cup N_G(v))$ are arbitrary vertices. Observe that $\delta(G^+_{uw}) = d_v = \delta(G)$ holds for every edge $uw \in E(G[L_G])$. 

\begin{claim} \label{cl::Lnoedges}
$E(G[L_G]) = \emptyset$. 
\end{claim}

\begin{proof}
Suppose for a contradiction that $w_1 w_2$ is an edge of $G[L_G]$. Suppose first that $w_1$ and $w_2$ are both internal and thus $L_{G^+ _ {w_1 w_2}} = L_G$ and $S_{G^+ _{w_1 w_2}} = S_G$. It follows that $X_{G^+_{w_1 w_2}} > X_G$, $Y_{G^+_{w_1 w_2}} = Y_G$ and $Z_{G^+_{w_1 w_2}} = Z_G$. Therefore $P_1(G^+_{w_1 w_2}, v) > P_1(G, v)$, contrary to the assumed maximality of $P_1(G, v)$. Next, suppose that $w_1$ and $w_2$ are both borderline. Then
\begin{eqnarray*}
e^{c-5} \left[P_1(G^+ _{w_1 w_2}, v) - P_1(G, v) \right] = M + M(w_1) + M(w_2), 
\end{eqnarray*}
where
$$
M = f(c - x_{w_1} + k/n) f(c - x_{w_2} + k/n),
$$
\begin{eqnarray*}
M(w_1) &=& \sum_{\stackrel{z \in L_G \setminus \{w_1, w_2\}}{(w_1,z) \in A_v}} \left[e^{-1} f(c - x_{w_1} + k/n) - e^{-2} \right] \\ 
&+& \sum_{\stackrel{z \in S_G}{(w_1,z) \in A_v}} \left[f(c - x_z) f(c - x_{w_1} + k/n) - e^{-1} f(c - x_z) \right],
\end{eqnarray*}
and
\begin{eqnarray*}
M(w_2) &=& \sum_{\stackrel{z \in L_G \setminus \{w_1, w_2\}}{(w_2,z) \in A_v}} \left[e^{-1} f(c - x_{w_2} + k/n) - e^{-2} \right] \\ 
&+& \sum_{\stackrel{z \in S_G}{(w_2,z) \in A_v}} \left[f(c - x_z) f(c - x_{w_2} + k/n)  - e^{-1} f(c - x_z) \right].
\end{eqnarray*}

Since $w_1$ is borderline, it follows that $x_{w_1} > c - 1$ and $x_{w_1} - k/n \leq c - 1$, or, equivalently, $1 \leq c - x_{w_1} + k/n < 1 + k/n$. It thus follows by Observation~\ref{obs:fx} that $e^{-1} = f(1) \geq f(c - x_{w_1} + k/n) \geq f(1 + k/n) = (1+o(1)) e^{-1}$. Since, moreover, $w_2$ is borderline, we conclude that $M = (1+o(1)) e^{-2}$.    

Now, let $z \in L_G \setminus \{w_1, w_2\}$ be an arbitrary vertex such that $(w_1,z) \in A_v$. Then
\begin{equation*}
e^{-2} - e^{-1} f(c - x_{w_1} + k/n) \leq e^{-2} - e^{-1} f(1 + k/n) 
\leq e^{-2} \left[1 - (1 + k/n) (1 - k/n) \right] = O(n^{-2}),  
\end{equation*} 
where the first inequality holds 
by Observation~\ref{obs:fx}. A similar calculation shows that, for every $z \in S_G$ such that $(w_1,z) \in A_v$, we have 
\begin{eqnarray*}
f(c - x_z) \left[e^{-1} - f(c - x_{w_1} + k/n) \right] = O(n^{-2}).  
\end{eqnarray*}
Since, trivially, $|L_G| = O(n)$ and $|S_G| = O(n)$, we conclude that $M(w_1) = O(n^{-1})$. An analogous argument shows that $M(w_2) = O(n^{-1})$ as well. We conclude that
$$
P_1(G^+_{w_1 w_2}, v) > P_1(G, v) + e^{3-c}/2,
$$
contrary to the maximality of $P_1(G, v)$. 

The remaining case, where exactly one of the vertices $w_1$ and $w_2$ is borderline and the other is internal, can be treated similarly; we omit the straightforward details. We conclude that $E(G[L_G]) = \emptyset$ as claimed. 
\end{proof}   

Next, we will use Claim~\ref{cl::Lnoedges} to prove that, in fact, $L_G$ itself is empty. 

\begin{claim} \label{cl::LGempty}
$L_G = \emptyset$. 
\end{claim}

\begin{proof}
Suppose for a contradiction that $L_G$ is not empty and consider the following switching operation on $G$. Given vertices $w \in L_G$ and $u_1, u_2 \in S_G$ such that $u_1 u_2 \notin E(G)$, $u_1 w \in E(G)$ and $u_2 w \in E(G)$, let $G' = \sigma_{u_1, u_2,w}(G)$ be any graph which is obtained from $G$ by deleting the edges $u_1 w$ and $u_2 w$ and adding the edges $u_1 u_2$, $wy_1$ and $wy_2$ for some vertices $y_1, y_2 \in V(G) \setminus (\{v\} \cup N_G(v))$. Observe that $\delta(G') = d_v = \delta(G)$.   

Note that, under the assumption that $L_G$ is not empty, a switching operation exists. This follows from the fact that, by definition, $x_z > x_{z'}$ for every $z \in L_G$ and $z' \in S_G$, and from Claim~\ref{cl::Lnoedges}. Let $G' = \sigma_{u_1, u_2, w}(G)$ be such a switch. 
If $w \in L_{G'}$ then we obtain  
\begin{align*}
e^{c-5} \left[P_1(G', v) - P_1(G, v) \right] &= e^{-1}f(c - x_{u_1})+ e^{-1}f(c - x_{u_2}) -  f(c - x_{u_1})f(c - x_{u_2})\\
&= e^{-2}-(e^{-1}-f(c-x_{u_1}))(e^{-1}-f(c-x_{u_2})).
\end{align*}
\noindent
Since $1\leq c-x_{u_1}\leq 2$, by Observation~\ref{obs:fx} we have $f(c-x_{u_1})\geq 2e^{-2}$, and analogously $f(c-x_{u_2})\geq 2e^{-2}$. Thus 
$$e^{c-5} \left[P_1(G', v) - P_1(G, v) \right]\geq e^{-2}(1-(1-2/e)^2)> 0.$$ 

Similarly, if $w\in S_{G'}$, we obtain
$$
e^{c-5} \left[P_1(G', v) - P_1(G, v) \right] = M + M(w)
$$
where 
\begin{align*}
M &=f(c-x_w+2k/n)f(c - x_{u_1})+ f(c-x_w+2k/n)f(c - x_{u_2}) -  f(c - x_{u_1})f(c - x_{u_2})\\
&=(1+o(1))e^{-1}f(c - x_{u_1})+ (1+o(1))e^{-1}f(c - x_{u_2}) -  f(c - x_{u_1})f(c - x_{u_2})\\
&> \frac{1}{2}e^{-2}(1-(1-2/e)^2),
\end{align*}
and
\begin{eqnarray*}
M(w) &=& \sum_{\stackrel{z \in L_G}{(w, z) \in A_v}} e^{-1} \left[f(c - x_{w} + 2k/n) - e^{-1} \right] \\ 
&+& \sum_{\stackrel{z \in S_G}{(w, z) \in A_v}} f(c - x_z) \left[f(c - x_{w} + 2k/n) - e^{-1} \right].
\end{eqnarray*}
A similar calculation to the one used for $M(w_1)$ in Claim~\ref{cl::Lnoedges} shows that $M(w) = O(n^{-1}) = o(1)$. 
We conclude that, in either case, $P_1(G',v) > P_1(G,v)$, contrary to the maximality of $P_1(G,v)$. 
\end{proof}
\noindent
Since $L_G = \emptyset$ by Claim~\ref{cl::LGempty}, it follows that $S_G = N_G(v)$ and thus $P_1(G,v)$ becomes 
$$
P_2(G,v) := \frac{1}{2} e^{5-c} \sum_{(u,w) \in A_v} f(c - x_u) f(c - x_w).
$$ 
\noindent
We will treat this case in the following lemma. 
\end{proof}

\begin{lemma} \label{lem:program3}
Let $G = (V,E)$ be a graph on $n$ vertices, satisfying $n/k \leq \delta(G) < 2n/k$. Let $v \in V$ be a vertex of minimum degree. Suppose that $c - x_u \geq 1$ holds for every $u \in N_G(v)$. Then 
\begin{equation*} 
P_2(G,v) \leq \frac{n^2}{k^2} \cdot \frac{128e}{81}.
\end{equation*}
\end{lemma}

\begin{proof}
Let $G$ be a graph and let $v \in V(G)$ be a vertex such that $P_2(G, v)$ is maximal among all pairs $(G, v)$ which satisfy all the conditions of Lemma~\ref{lem:program3}. Let $F = \overline{G}[N_G(v)]$ and note that (for a fixed $k$) the quantity $P_2(G,v)$ is a function of $F$ and $n$. Hence, we may write $P(F,n) = P_2(G,v)$ and assume that $F$ maximises $P(F,n)$ amongst all eligible pairs. For every $u \in V(F)$ let $z_u = c - x_u$. Using this notation we can write 
$$ 
P_2(G,v) = P(F,n) = \frac{1}{2} e^{5-c} \sum_{(u,w) : uw \in E(F)} f(z_u) f(z_w).
$$ 

For every $u \in V(F)$, let $\bar{N}_F(u) := N_F(u) \cup \{u\}$ and let 
\begin{equation} \label{eq::yu}
y_u = \frac{1}{|\bar{N}_F(u)|} \cdot \sum_{w \in \bar{N}_F(u)} z_w = \frac{k}{n} \cdot \frac{1}{z_u} \sum_{w \in \bar{N}_F(u)} z_w,
\end{equation}
where the second equality holds since $|\bar{N}_F(u)| = z_u n/k$. 

Using the right hand side of~\eqref{eq::yu}, it is not hard to verify that $\sum_{u \in V(F)} z_u y_u = \sum_{u \in V(F)} z_u^2$. Since, by Observation~\ref{obs:fx}, the function $f$ is concave on $[1,2]$, it follows by Jensen's inequality that
\begin{align*}
P(F,n) &\leq \frac{1}{2} e^{5-c} \sum_{u \in V(F)} f(z_u) \sum_{w \in \bar{N}_F(u)} f(z_w) \leq \frac{1}{2} e^{5-c} \sum_{u \in V(F)} f(z_u) \cdot |\bar{N}_F(u)| f(y_u) \\
&= \frac{1}{2} e^{5-c} \sum_{u \in V(F)} f(z_u) \cdot \frac{n}{k} z_u \cdot f(y_u) = \frac{1}{2} e^{5-c} \frac{n}{k} \sum_{u \in V(F)} z_u^2 y_u e^{- z_u - y_u}.
\end{align*}

In order to bound $P(F,n)$ from above, we will now completely abandon the graph structure and analyse the function at hand under more general conditions. Given any positive integer $m$ and any real number $1 \leq c \leq 2$, let $A(c,m)$ denote the following optimisation problem: maximise $\sum_{i=1}^m z_i^2 y_i e^{- z_i - y_i}$ subject to the constraints $1 \leq y_i, z_i \leq c$ for every $1 \leq i \leq m$, (note that it is here that we use the assumption $c-x_u\geq 1$) and $\sum_{i=1}^m z_i^2 = \sum_{i=1}^m y_i z_i$. 


\begin{claim} \label{cl:programA}
$y_i = z_i = \min\{c, 3/2\}$ for every $1 \leq i \leq m$, is a solution to the optimisation problem $A(c,m)$.
\end{claim}
The proof of Claim~\ref{cl:programA} is a tedious yet straightforward exercise in multivariate calculus and is thus presented in the Appendix.

With Claim~\ref{cl:programA} at our disposal, we can now conclude the proof of Lemma~\ref{lem:program3} as follows. Let $z = \min\{c, 3/2\}$. Since the maximum of the optimisation problem $A(c, |F|)$ is achieved when $y_i = z_i = z$ for every $1 \leq i \leq |F|$ and since $|F| = c n/k$, we have
\begin{equation*}
P(F,n) \leq \frac{1}{2} e^{5-c} \frac{n}{k} \cdot |F| z^3 e^{-2z} = \frac{n^2}{k^2} \cdot \frac{1}{2} e^5 z^3 e^{-2z} c e^{-c}.
\end{equation*}
\noindent
Since $c \geq z \geq 1$ by Observation~\ref{obs:fx}, we have
$$
P(F,n) \leq \frac{n^2}{k^2} \cdot \frac{1}{2} z^4 e^{5-3z}.
$$
By differentiating, it is easy to see that the last expression attains its global maximum at $z = 4/3$, yielding
$$
P_2(G,v) \leq \frac{n^2}{k^2} \cdot \frac{1}{2} \left(\frac{4}{3}\right)^4 e^{5-4} = \frac{n^2}{k^2} \cdot \frac{128e}{81},
$$
as claimed.
\end{proof}

\section*{Appendix}


\begin{proof}[Proof of Claim~\ref{cl:programA}] 
We distinguish between two cases according to the value of $c$.

\medskip

\noindent \textbf{Case 1: $c \geq 3/2$}. It suffices to consider the optimisation problem $A(2,m)$. Note that the maximum is always attained, as we are dealing with a continuous function over a compact domain. Let $(y_1, \ldots, y_m, z_1, \ldots ,z_m)$ be a point which attains the maximum; our aim is to show that $y_i = z_i = 3/2$ for every $1 \leq i \leq m$.

\begin{claim}\label{cl:summary}
For every $1 \leq i \leq m$, either $y_i=z_i=3/2$ or $y_i < 3/2 < z_i$ or $z_i \leq 3/2 <y_i$.
\end{claim}

\begin{proof}
Fix some $1 \leq i \leq m$. To simplify notation, at the moment we abbreviate $y_i$ to $y$ and $z_i$ to $z$. Let $\Delta = z^2 - yz$, and observe that $-1 \leq \Delta \leq 2$. Our current aim is to maximise $y z^2 e^{- y - z}$ subject to $1 \leq y, z \leq 2$ and $z^2 - yz = \Delta$.

Suppose first that $(y,z)$ lies at the boundary of $[1,2]^2$. Since $y = z - \Delta/z$, we can write  
$$
y z^2 e^{- y - z} = z^2 e^{-z} \left(z - \Delta/z \right) e^{- \left(z - \Delta/z \right)}:=h(z, \Delta).
$$
For a fixed $\Delta$, differentiating with respect to $z$ yields 
\begin{align}\label{eq::deltaz}
\frac{\partial h}{\partial z}(z, \Delta) = h(z, \Delta) \cdot \left(\frac{2}{z} - 1 + \frac{1 + \Delta/z^2}{z - \Delta/z} - (1 + \Delta/z^2)\right).
\end{align}
Suppose that $z = 2$. Then~\eqref{eq::deltaz} implies
\begin{equation} \label{eq::z2}
\frac{\partial h}{\partial z}(2, \Delta) = h(2, \Delta) \cdot \left(1 + \frac{\Delta}{4}\right) \left(\frac{1}{2 - \Delta/2} - 1\right) \leq 0,
\end{equation}
where the last inequality is strict unless $\Delta = 2$ or, equivalently, $y = 1$. We claim that indeed $z = 2$ entails $y = 1$. Suppose to the contrary that $y > 1$. Since the one-sided derivative~\eqref{eq::z2} is negative, there exist $1 < y_0 < y$ and $1 < z_0 < 2$ such that $z_0^2 - y_0 z_0 = \Delta$, but $h(z_0, \Delta) > h(2, \Delta)$ contrary to the assumed maximality of $h(2, \Delta)$.    

Next, suppose that $z=1$. Then, in particular, $\Delta \leq 0$ and by~\eqref{eq::deltaz} we have
$$
\frac{\partial h}{\partial z}(1, \Delta) = h(1, \Delta) \cdot \left(2 - 1 + \frac{1 + \Delta}{1 - \Delta} - 1 - \Delta \right) = h(1, \Delta) \cdot \frac{1 + \Delta^2}{1 - \Delta} > 0,
$$
i.e., the one-sided derivative is positive for every $y$. Hence, an analogous argument to the one used above for the previous case, shows that we can only have $z=1$ if $y=2$.

Next suppose that $y = 1$. Recalling that $y = z - \Delta/z$, by~\eqref{eq::deltaz} we obtain
$$
\frac{\partial h}{\partial z}(z, \Delta) = h(z, \Delta) \cdot \left(\frac{2}{z} - 1\right) \geq 0,
$$
where the last inequality is strict whenever $z < 2$. Hence, $y = 1$ entails $z = 2$.

Lastly, suppose that $y = 2$. Then, again by~\eqref{eq::deltaz}, we get
$$
\frac{\partial h}{\partial z}(z, \Delta) = h(z, \Delta) \cdot \left(\frac{2}{z} - 1 - \frac{1 + \Delta/z^2}{2}\right) = \frac{1}{2z^2} h(z, \Delta) \cdot (4z - 3z^2 - \Delta).
$$
Since $\Delta = z^2 - yz = z^2 - 2z$, we obtain
$$
\frac{\partial h}{\partial z}(z, \Delta) = \frac{1}{2z^2} h(z, \Delta) (6z - 4 z^2) = \frac{1}{z} h(z, \Delta) (3 - 2z),
$$
which is non-negative if and only if $z \leq 3/2$. Hence, $y = 2$ entails $z \leq 3/2$. 

Now, suppose that $1 < y, z < 2$. In this case, by the method of Lagrange multipliers, we infer the existence of a constant $\lambda\in \mathbb{R}$ satisfying
\begin{equation} \label{eq::lambday}
(1-y) z^2 e^{- z - y} = - \lambda z
\end{equation}
and 
\begin{equation} \label{eq::lambdaz}
(2z - z^2) y e^{- z - y} = \lambda (2z - y).
\end{equation}

Since $1 < y, z < 2$ by assumption, it follows that $2z - y \neq 0$. Hence, rearranging~\eqref{eq::lambday} and~\eqref{eq::lambdaz}, we obtain 
$$
\lambda = e^{- z - y} z (y-1),
$$
and 
$$
\lambda = e^{- z - y} z \left(\frac{2y - zy}{2z - y}\right)
$$
respectively. Therefore $(y-1) (2z-y) = 2y - zy$. Since, moreover, $3y - 2 \neq 0$, it follows that
\begin{align} \label{eq:zylagrange}
z = \frac{y^2 + y}{3y - 2}.
\end{align}
Straightforward calculations then show that
$$
\begin{cases}
z > 3/2 &\textrm{ if } y < 3/2 \\ 
z = 3/2 &\textrm{ if } y = 3/2 \\
z < 3/2 &\textrm{ if } y > 3/2.
\end{cases}
$$
This concludes the proof of Claim~\ref{cl:summary}.
\end{proof}
Next, suppose that there are $1 \leq i, j \leq m$ such that $y_i < z_i$ and $y_j > z_j$; note that due to the constraint $\sum_{i=1}^m z_i^2 = \sum_{i=1}^m y_i z_i$, such indices must exist unless $y_i = z_i$ for every $1 \leq i \leq m$. By Claim~\ref{cl:summary} we infer that $y_i < 3/2 < z_i$ and $y_j > 3/2 \geq z_j$, which implies that $y_i < y_j$ and $z_i > z_j$. Let $y'_i = y_i + \varepsilon_i$ and $y'_j = y_j - \varepsilon_j$ where $0 < \varepsilon_i, \varepsilon_j \ll 1$ are chosen such that $\varepsilon_i z_i = \varepsilon_j z_j$. Then 
$$
z_i^2 + z_j^2 - y'_i z_i - y'_j z_j = z_i^2 + z_j^2 - y_i z_i - y_j z_j - \varepsilon_i z_i + \varepsilon_j z_j = z_i^2 + z_j^2 - y_i z_i - y_j z_j. 
$$
Hence, by replacing $y_i$ with $y'_i$ and $y_j$ with $y'_j$ we do not violate the constraints in the optimisation problem $A(2,m)$. Using Taylor's expansions $e^{- \varepsilon_i} = 1 - \varepsilon_i + O(\varepsilon_i^2)$ and $e^{\varepsilon_j} = 1 + \varepsilon_j + O(\varepsilon_j^2)$, the change in the objective function is seen to be 
\begin{align*}
& y'_i e^{- y'_i} z_i^2 e^{-z_i} + y'_j e^{-y'_j} z_j^2 e^{-z_j} -
(y_i e^{-y_i} z_i^2 e^{-z_i} + y_j e^{-y_j} z_j^2 e^{-z_j}) \\
&= \varepsilon_i (1 - y_i) e^{-y_i} z_i^2 e^{-z_i} - \varepsilon_j (1 - y_j) e^{-y_j} z_j^2 e^{-z_j} + o(\varepsilon_i) + o(\varepsilon_j).
\end{align*}
Using the identity $\varepsilon_i z_i = \varepsilon_j z_j$, the right hand side of the above can be written as
\begin{align*}
\varepsilon_i z_i \cdot \left[(y_j - 1) e^{- y_j} z_j e^{- z_j} - (y_i - 1) e^{- y_i} z_i e^{- z_i} \right]+ o(\varepsilon_i) + o(\varepsilon_j).
\end{align*}
Note that, for $1 \leq x \leq 2$, the function $(x-1) e^{-x}$ is non-negative and strictly increasing, and the function $x e^{-x}$ is non-negative and strictly decreasing. Since, moreover, $y_j >  y_i$ and $z_j < z_i$, it follows that 
$$
(y_j - 1) e^{-y_j} z_j e^{-z_j} - (y_i - 1) e^{-y_i} z_i e^{-z_i} > 0.
$$
Hence 
$$
y'_i e^{-y'_i} z_i^2 e^{-z_i} + y'_j e^{-y'_j} z_j^2 e^{-z_j} - (y_i e^{-y_i} z_i^2 e^{-z_i} + y_j e^{-y_j} z_j^2 e^{-z_j}) > 0,
$$
increasing the value of the objective function, contrary to the assumed maximality.
 
Therefore, the only way the maximum can be attained is when $y_i = z_i$ for every $1 \leq i \leq m$, which, by Claim~\ref{cl:summary}, can only happen when $y_i = z_i = 3/2$ for every $1 \leq i \leq m$.  

\medskip

\noindent \textbf{Case 2: $c < 3/2$}. The proof of this case is fairly similar to that of Case 1. As before, let $(y_1, \ldots, y_m, z_1, \ldots ,z_m)$ be a point which attains the maximum of $A(c, m)$; our aim is to show that $y_i = z_i = c$ for every $1 \leq i \leq m$. The following claim and its proof (which is where we use the fact that $c < 3/2$) are analogous to Claim~\ref{cl:summary} in Case 1. We omit the details.
\begin{claim} \label{cl:atleastonec}
$y_i = c$ or $z_i = c$ holds for every $1 \leq i \leq m$.
\end{claim}
%

Next, suppose that there are $1 \leq i, j \leq m$ such that $y_i < z_i$ and $y_j > z_j$; note that due to the constraint $\sum_{i=1}^m z_i^2 = \sum_{i=1}^m y_i z_i$, such indices must exist unless $y_i = z_i$ for every $1 \leq i \leq m$. By Claim~\ref{cl:atleastonec}, it follows that $y_i < c, z_i = c$ and $y_j = c, z_j < c$.

Similarly to Case 1, let $y'_i = y_i + \varepsilon_i$ and $y'_j = c - \varepsilon_j$, where $0 < \varepsilon_i, \varepsilon_j \ll 1$ are chosen such that $\varepsilon_i c = \varepsilon_j z_j$. Hence 
$$
z_i^2 + z_j^2 - z_i y_i - z_j y_j = z_i^2 + z_j^2 - z_i y'_i - z_j y'_j.
$$ 
In other words, by replacing $y_i$ with $y'_i$ and $y_j$ with $y'_j$ we do not violate the constraints in the optimisation problem $A(c, m)$. Similarly to Case 1, the change in the objective function is
\begin{align*}
& y'_i e^{- y'_i} z_i^2 e^{-z_i} + y'_j e^{-y'_j} z_j^2 e^{-z_j} -
(y_i e^{-y_i} z_i^2 e^{-z_i} + y_j e^{-y_j} z_j^2 e^{-z_j}) \\
&= \varepsilon_i (1 - y_i) e^{-y_i} z_i^2 e^{-z_i} - \varepsilon_j (1 - y_j) e^{-y_j} z_j^2 e^{-z_j} + o(\varepsilon_i) + o(\varepsilon_j).
\end{align*}
As in Case 1, using the identity $\varepsilon_i c = \varepsilon_j z_j$, the right hand side of the above can be written as
\begin{align*}
\varepsilon_i c e^{- c} \cdot \left[(c - 1) z_j e^{- z_j} - (y_i - 1) e^{- y_i} c \right]+ o(\varepsilon_i) + o(\varepsilon_j).
\end{align*}

Similarly to Case 1, for $1 \leq x \leq c$, the function $(x-1) e^{-x}$ is strictly increasing and the function $x e^{-x}$ is strictly decreasing. Since, moreover, $y_i < c$ and $z_j < c$, it follows that 
$$
\varepsilon_i c e^{-c} \cdot[(c - 1) z_j e^{- z_j} - (y_i - 1) e^{- y_i} c] > 
\varepsilon_i c e^{-c} \cdot[(c - 1) c e^{-c} - (c - 1) e^{-c} c] = 0.
$$
Therefore, 
$$
y'_i e^{- y'_i} z_i^2 e^{-z_i} + y'_j e^{- y'_j} z_j^2 e^{- z_j} - (y_i e^{- y_i} z_i^2 e^{- z_i} + y_j e^{- y_j} z_j^2 e^{-z_j}) > 0,
$$
increasing the value of the objective function, contrary to the assumed maximality.
 
Hence, the only way the maximum can be attained is when $y_i = z_i$ for every $1 \leq i \leq m$, which, by Claim~\ref{cl:atleastonec}, can only happen when $y_i = z_i = c$ for every $1 \leq i \leq m$.  
\end{proof}

\end{document}